\documentclass[reqno,a4paper,12pt]{amsart}

\usepackage[full]{textcomp}
\usepackage{fourier}
\usepackage{latexsym}

\usepackage{amssymb,enumerate}

\usepackage{graphics}

\usepackage[backref=page]{hyperref} 
\usepackage{amsmath,hhline, graphics,
graphicx, verbatim, amssymb, caption, xspace, fancyhdr,longtable, array, }
\usepackage[pdf]{pstricks}
\usepackage{pst-node,pst-tree}
\usepackage{pst-plot}

\newcommand{\NN}{{\mathbb{N}}}

\newcommand{\A}{{\mathcal{A}}}

\newcommand{\V}{\mathcal{V}}

\newcommand{\FF}{\mathfrak{F}}
\newcommand{\FQ}{{\mathfrak{F}_{_Q}}}

\newcommand{\T}{{\mathcal{T}}}

\newcommand{\kk}{{\mathbf k}}
\newcommand{\XX}{{\mathbf X}}

\def\kall{{(\!(\mathbf k)\!)}}

\newcommand{\norm}[1]{\left|\hspace*{2.5pt} \!\!\left|
#1\right|\hspace*{2.5pt} \!\!\right|}

\numberwithin{equation}{section}
\numberwithin{figure}{section}

\newcolumntype{C}{>{$}c<{$}}
\newcolumntype{L}{>{$}l<{$}}
\newcolumntype{R}{>{$}r<{$}}
\newtheorem{theorem}{Theorem}
\newtheorem{lemma}{Lemma}
\newtheorem{proposition}{Proposition}

\theoremstyle{remark}
\newtheorem{remark}{Remark}[section] 

\begin{document}

\title[Geometry behind a recurrent relation]
{On the geometry behind a recurrent relation}
\author{Cristian Cobeli and Alexandru Zaharescu}

\address{
\indent CC \textit{and} AZ: 
\emph{Simion Stoilow},
 Institute of Mathematics of the Romanian Academy,  \newline
\indent
P.O. Box 1-764, 
Bucharest, 70700, Romania.}
\email{cristian.cobeli@imar.ro}

\address{
\indent AZ:
Department of Mathematics, 
University of Illinois at Urbana - Champaign,  \newline
\indent 1409 West Green Street, 
Urbana, IL 61801, USA.
}
\email{zaharesc@illinois.edu}

\subjclass[2010]{Primary 11B57, 
Secondary 11B37}
\keywords{\em Farey sequences, tessellations, geometry of recurrent relations.}
\begin{abstract}
We consider a certain linear recursive relation with integer parameters and
study some of its algebraic and geometric properties, with the purpose
of estimating the number of chains of valences in the Farey series.
\end{abstract}

\maketitle

\section{Introduction}\label{sec:Introduction}

In the present paper we study some of the significant properties and consequences of a recurrent
construction involving a sequence of polynomials that appears naturally in different contexts where
the
algebra and geometry are linked with the arithmetical features of integers. As in \cite{CZ2006},
let $p_{-1}(\cdot)=0$, $p_0(\cdot)=1$, and then recursively, for any
integer $r\ge 1$ and variables $X_1,X_2,\dots X_r$, let
    \begin{equation}\label{eqR}
	p_r(X_1,\dots,X_r)=X_rp_{r-1}(X_1,\dots,X_{r-1})
			-p_{r-2}(X_1,\dots,X_{r-2}).\tag{R}
    \end{equation}
We write $\XX=(X_1,\dots,X_r)$ when the \emph{order} (or
\emph{length}) $r$ is understood from the context and $p_r(\XX)$
instead of $p_r(X_1,\dots,X_r)$. 
A few polynomials of small orders satisfying \eqref{eqR} are:
   \begin{align*}
     p_1(\XX)=&X_1;\qquad \quad
     p_2(\XX)=X_1X_2-1;\qquad \quad
     p_3(\XX)=X_1X_2X_3-X_1-X_3;\\
     p_4(\XX)=&X_1X_2X_3X_4-X_1X_2-X_1X_4-X_3X_4+1;\\
     p_5(\XX)=&X_1X_2X_3X_4X_5-X_1X_2X_3-X_1X_2X_5
                  -X_1X_4X_5-X_3X_4X_5+X_1+X_3+X_5\\
     p_6(\XX)=&X_1X_2X_3X_4X_5X_6 - X_1X_2X_3X_6-X_1X_4X_5X_6 \\
&\phantom{X_1X_2X_3X_4X_5X_6}- X_1X_2X_5X_6
	   -X_1X_2X_5X_6-X_1X_2X_3X_4-X_3X_4X_5X_6\\
&+X_1X_2+X_5X_6+X_1X_4+X_3X_6+X_1X_6+X_3X_4-1\,.
   \end{align*}  
Relation \eqref{eqR} has many nice properties.  
For example,
it produces the symmetry 
\begin{equation}\label{eqSimmetry}
   p_r(X_1,\dots,X_r)=p_r(X_r,\dots,X_1)\,.
\end{equation}
Notice also  the alternation in the signs of  the monomials of $p_r(\XX)$ for values of $r$ of the
same parity.
The polynomials $p_r(\XX)$ will be used with suitable values
$X_j=k_j$.
We call an $r$-tuple $\kk=(k_{1},\dots,k_{r})$ of positive integers \textit{admissible} 
if there exists an integer $Q\ge 1$ and integers $1\le
q_0,q_1,\dots,q_{r+1}\le Q$ with the following properties:
     \begin{subequations}
       \begin{eqnarray}
         \gcd(q_j,q_{j+1})=1,\quad \text{ for $0\le j\le r$;}
         \label{defa}\\
         q_j+q_{j+1}>Q, \quad \text{ for $0\le j\le r$;}
         \label{defb}\\
         k_jq_j=q_{j-1}+q_{j+1},\quad \text{ for $1\le j\le r$.}
         \label{defc}
     \end{eqnarray}
   \end{subequations}
We call the components of $\kk$ \emph{valences} and say that they
are generated by the \emph{denominators} $q_0, q_1,\dots,q_{r+1}$.
An $r$-tuple of consecutive valences will also be called a
\emph{chain of valences}.

\begin{center}
\begin{pspicture}(-2,-0.5)(10,2.5)%
$
\psset{nodesep=2pt}
\rput(-1,0){\rnode{ALEF}{q_0}}
\rput(0,0){\rnode{A}{q_1}}
\rput(1,0){\rnode{B}{q_2}}
\rput(2,0){\rnode{C}{q_3}}
\rput(3,0){\rnode{D}{\ }}
\rput(4,0){\Rnode{E}{\ }}
\rput(5,0){\rnode{E1}{\cdots\cdots\cdots\cdots\cdots\cdots\cdots}}
\rput(0,2){\rnode{a}{k_{1}}}
\rput(1,2){\rnode{b}{k_{2}}}
\rput(2,2){\rnode{c}{k_{3}}}
\rput(3,2){\rnode{d}{\ }}
\rput(4,2){\rnode{e}{\ }}
\rput(5,2){\rnode{e1}{\cdots\cdots\cdots\cdots\cdots\cdots\cdots}}
\rput(7,2){\rnode{s}{\ }}
\rput(8,2){\rnode{t}{k_{r-1}}}
\rput(7,0){\rnode{T}{\ }}
\rput(9,2){\rnode{y}{k_{r}}}
\rput(8,0){\rnode{X}{q_{r-1}}}
\rput(9,0){\rnode{Y}{q_{r}}}
\rput(10,0){\rnode{Z}{q_{r+1}}}
\ncarc{<->}{ALEF}{a}
\ncline{<->}{a}{A}
\ncarc{<->}{a}{B}
\ncarc{<->}{A}{b}
\ncline{<->}{b}{B}
\ncarc{<->}{b}{C}
\ncarc{<->}{B}{c}
\ncline{<->}{c}{C}
\ncarc{<->}{c}{D}
\ncarc{<->}{C}{d}
%
\ncarc{<->}{T}{t}
\ncline{<->}{t}{X}
\ncarc{<->}{t}{Y}
\ncarc{<->}{s}{X}
\ncarc{<->}{T}{t}
\ncline{<->}{t}{X}
\ncarc{<->}{t}{Y}
\ncarc{<->}{X}{y}
\ncline{<->}{y}{Y}
\ncarc{<->}{y}{Z}
\rput(-0.5,-1){\rnode{M}{}}
\psellipse[linestyle=dotted](-0.5,0)(1,0.6)
$
\end{pspicture}
\newline
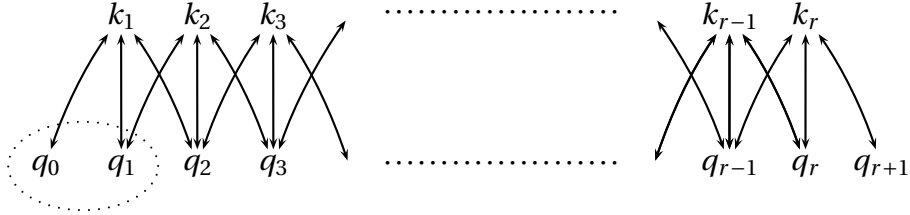
{\small {Figure 1.} A chain of valences and their generators.}
\end{center}

Notice that it would be enough to only require in \eqref{defa} that two neighbor
denominators are relatively prime, since by \eqref{defc} the same
property radiates recursively to all the other pairs of neighbor denominators.

We remark that by relations \eqref{defa}-\eqref{defc} it follows that any admissible $r$-tuple
$\kk$ can be extended  to an admissible  sequence $\kall$ that is infinite on both ends. (We call
an infinite sequence admissible if all its $r$-subchains of consecutive valences are admissible.)
Notice that the extension is not unique.
There is a close  connection between the sequence of polynomials defined by relation \eqref{eqR}
and Farey sequences.
For more details the reader is referred to \cite[Section 6]{CZ2006}.

A few experiments reveal a peculiar property of $\kall$. One may find
 in $\kall$ components indefinitely large, but in any neighborhood
of such a component all the others are comparatively small. And the
larger a valence is, the larger is its neighborhood with only small
components. Here are a few examples of admissible chains $\kk$, that shed some light on this
phenomenon:
\begin{equation*}
  \begin{split}
  &[11,1,2,2,3,1,5,1,3,2,2,1,12];\\
  &[10,1,2,3,1,5,1,4,1,3,2,2,1,15];\\
  &[16,1,2,2,2,3,1,5,1,3,1,7,1,3,1,5,1,3,2,2,2,1,16];\\
  &[1,6,1,3,1,5,1,4,1,3,2,2,2,2,2,2,1,28,1,2,2,2,2,2,3].
  \end{split}
\end{equation*}
 
For any chain of valences $\kk$, we define the norm of $\kk$, 
to be its largest component.
We denote the norm of $\kk$ by  $\norm{\kk}$.
Let $\A_r$ be the set of admissible chains of valences of
\mbox{length $r$}.  Our aim is to
estimate the size of $\A_r$. The main result below unveils the following peculiar fact:
for each
positive integer $r$,
 the number of admissible chains of length $r$ and norm at most $x$ grows almost linearly as a
function of $x$.

\begin{theorem}\label{TheoremAr}
For any integer $r\ge 1$, we have
    \begin{equation}\label{eqFormula}
      \sum_{\substack{\kk\in\A_r\\ \norm{\kk}\le x}}1 
      = rx+O_r(1).
    \end{equation}
\end{theorem}

\begin{remark}\rm
We found that for $n$ positive integer and sufficiently large the difference
\begin{equation*}
   \delta_r(n):=\#\big\{
	\kk\in\A_r\colon\ \norm{\kk} \le n
	   \big\}
	-rn
 \end{equation*}
 becomes constant.
We denote this constant, which depends only on $r$, by $C(r)$. 
The first twenty five values of $C(r)$ are:
$C(1)=0;$
$C(2)=3;$
$C(3)=15;$
$C(4)=41;$
$C(5)=84;$
$C(6)=153;$
$C(7)=247;$
$C(8)=367;$
$C(9)=523;$
$C(10)=721;$
$C(11)=961;$
$C(12)=1251;$
$C(13)=1588;$
$C(14)=1983;$
$C(15)=2437;$
$C(16)=2963;$
$C(17)=3548;$
$C(18)=4219;$
$C(19)=4954;$
$C(20)=5761.$
\end{remark}

\noindent
\textbf{Open problem.} 
We leave open the question of whether there exists a closed formula for $C(r)$ for all $r$, or at
least for $r$ large enough.

\section{The Farey Sequence}
\label{sec:Farey}
About two
hundred years ago Haros and Farey observed (see also~\cite{BA1995}, \cite{CZ2003} and the references
therein)  that by arranging
the subunitary fractions with denominators at most a given $Q\ge 1$ in
ascending order, the finite sequence obtained has remarkable
properties. Thus, if $a'/q'<a''/q''$ are consecutive fractions then
$a''q'-a'q''=1$ and $q'+q''>Q$.
Given $Q\ge 1$, let 
\begin{equation*}
\FQ:=\left\{\frac aq\in [0,1]\colon\ \gcd(a,q)=1,\ q\le Q\right\}.
\end{equation*}
Arranged in ascending order, this is the Farey sequence of order $Q$. 
For example the sequence of Farey fractions of order $8$ is
\begin{equation*}
  \FF_{_8}:=\left\{\frac 01, \frac 18, \frac 17,\frac 16, \frac 15, \frac 14, 
  \frac 27, \frac 13, \frac 38,
  \frac 25, \frac 37,\frac 12,\frac 47, \frac 35, \frac 58, \frac 23, \frac 57,
  \frac 34,\frac 45,\frac 56, \frac 67, \frac 78, \frac 11\right\}.
\end{equation*}
Relations \eqref{defa}-\eqref{defc} are crystallized from some basic features
of a Farey sequence. If
$a'/q',a''/q'',a'''/q'''$ are consecutive Farey fractions, then
$(a'+a''')/a''=(q'+q''')/q''\in\NN^*$, any neighbor denominators
are relatively prime, and their sum is greater than the order.
Thus, if $a_0/q_0,a_1/q_1,\dots,a_{r+1}/q_{r+1}$ are consecutive fractions in $\FQ$ and
$\kk=(k_1,\dots,k_r)$ is the chain
of the associated valences to the fractions, in \cite{CZ2006}
it is shown that
  \begin{equation*}
    \frac{a_{r+1}}{q_{r+1}}-\frac{a_0}{q_0}=\frac{p_r(\kk)}{q_0q_{r+1}}\,.
  \end{equation*}
In recent years the authors of
\cite{BCZ2000}, \cite{BCZ2001},
\cite{ABCZ2001},
\cite{BGZ2002},
  \cite{CFZ2003}, \cite{BCZ2003}, \cite{CIZ2003},
\cite{Hay2003},
\cite{Hay2004},
\cite{BZ2005},
\cite{BZ2006},
\cite{CZ2006}, 
\cite{Boc2008},
\cite{CVZ2010a},  
\cite{CVZ2010b}, 
\cite{BH2011},
\cite{CVZ2012},  
\cite{AC2013}
investigated various questions on the distribution of Farey fractions, the tessellations
and their polygonal  tiles.

\section{Germs, Tiles and  Tessellations}
\label{sec:Germs}

Any chain of valences
has many corresponding chains of denominators (actually infinitely
many as $Q\to\infty$), but conversely, exactly one chain of valences
corresponds to a given chain of admissible denominators.

In the following, if $(k_1,\dots,k_r)$ is a chain of valences with
denominators $(q_0,q_1,\dots,q_r)$, we shall call the pair $(q_0,q_1)$
a pair of \emph{integer germs} of $\kk$.

For a given $\kk\in\A_r$ and $Q$ sufficiently large, let
$\T^Q[\kk]=\T_r^Q[\kk]$ be the set of integer germs of $\kk$. Since
this set depends on $Q$ and we are interested in all admissible chains,
independent of the size of their germs, it is natural to let $Q$
approach infinity and move the problem to a bounded frame. 
Starting with two variables $x,y$, we put $x_{-1}=x$, $x_0=y$
and then define
    \begin{equation}\label{eqRR}
	x_j=x_j(k_1,\dots,k_j;x,y):=k_jx_{j-1}-x_{j-2},\quad
        \text{for $j\ge 1$,}\tag{RR}
    \end{equation}
where $k_j=\Big[ \tfrac{1+x_{j-2}}{x_{j-1}} \Big]$.
The connection with relation \eqref{eqR} is given in the next lemma.
(Observe also that when $x=0$ and $y=1$ \eqref{eqRR} produces the same
sequence as \eqref{eqR}.)
	\begin{lemma}\label{Lemma3}
	We have:
    \begin{equation}\label{eqL3}
	x_j=p_j(k_1,\dots,k_j)y-p_{j-1}(k_2,\dots,k_j)x,
		\qquad \text{\rm{for}  $j\ge 1$}\,.	
    \end{equation}
	\end{lemma}
	\begin{proof}
The proof is by induction. For $j=1$, relation
\eqref{eqL3} coincides with \eqref{eqRR}. For $j\ge 2$, using
\eqref{eqR} and the induction hypothesis, we have:
    \begin{equation*}
	\begin{split}
	& p_{j+1}(k_1,\dots,k_{j+1})y-p_{j}(k_2,\dots,k_{j+1})x\\
	=&\big(k_{j+1}p_{j}(k_1,\dots,k_{j})-p_{j-1}(k_1,\dots,k_{j-1})\big)y-	
	  \big(k_{j+1}p_{j-1}(k_2,\dots,k_{j})-p_{j-2}(k_2,\dots,k_{j-1})\big)x\\
	=&k_{j+1}(k_jx_{j-1}-x_{j-2})-(k_{j-1}x_{j-2}-x_{j-3})\\
	=&k_{j+1}x_j-x_{j-1}=x_j\,.
	\end{split}
    \end{equation*}
This completes the proof of the lemma.
	\end{proof}
The next result expresses the integers defined by \eqref{eqRR}
in the language of inequalities.

\begin{lemma}\label{Lemma4}
For any $j\ge 1$, the equality 
$ k_j=\Big[\tfrac{1+x_{j-2}}{x_{j-1}} \Big]$ is equivalent to
    \begin{equation}\label{eqL4}
	\frac{p_{j-1}(k_2,\dots,k_j+1)x+1}{p_j(k_1,\dots,k_j)}< y\le
	\frac{p_{j-1}(k_2,\dots,k_j  )x+1}{p_j(k_1,\dots,k_j)}\,.
    \end{equation}
	\end{lemma}
	\begin{proof}
The lemma follows by translating
$ k_j=\Big[\tfrac{1+x_{j-2}}{x_{j-1}} \Big]$ into
       \begin{equation*}
	\frac{1+x_{j-2}}{x_{j-1}}-1<k_j\le
	\frac{1+x_{j-2}}{x_{j-1}}\,,
       \end{equation*}
and inserting here the information provided by \eqref{eqL3}.
\end{proof}
The \emph{Farey triangle} is defined by 
  \begin{equation*}
    \T =\big\{ (x,y)\colon\ 
    x+y>1, \text{ and } 0<x,y\le1  \big\} \,,
  \end{equation*}
and
  \begin{equation}\label{eqtes}
    \T_r[k_1,\dots,k_r] =\big\{ (x,y)\in \T \colon\ 
    k_j=\Big[ \tfrac{1+x_{j-2}}{x_{j-1}} \Big],\ \text{for}\ 1\le j\le r  \big\} \,.
  \end{equation}
We call $\T_r[\kk]$ the \emph{tile} of $\kk$. 
Similarly we say that any pair $(x,y)\in\T_r[\kk]$ is a \emph{germ} of $\kk$.

Notice that one can write
   \begin{equation}\label{eqalgorit}
     \begin{split}
       \T_j[k_1,\dots,k_j]&=\Big\{(x,y)\in\T_{j-1}[k_1,\dots,k_{j-1}]\colon\ k_j
       =\Big[\frac{x+1}{y}\Big]\Big\},\quad\text{for }j\ge 1,
     \end{split}
   \end{equation}
with $\T_0[\cdot]:=\T$. This shows that any tile of any admissible
chain is a convex polygon.  It is easy to see that
any two tiles are disjoint, and the set of all tiles $\T_r[\kk]$ with
$\kk\in\A_r$ form a partition of $\T$, which we call the
\emph{tessellation} of order $r$. In this language, our main problem is
to estimate the number of tiles in such a tessellation.

The expression \eqref{eqalgorit} gives also an algorithm to find germs
of $\kk$: Calculate $\T_r[\kk]$; if it is empty, then $\kk$ is not
admissible.  Otherwise choose $Q$ sufficiently large and pick a
pair of relatively prime integers $(q_0,q_1)\in Q\T_r[\kk]$. This is
an integer germ of $\kk$ and any other germ is obtainable by this
method.  In conclusion, given a chain of valences $\kk$ and $Q$
sufficiently large, the polygon $Q\T[\kk]$ contains plenty generators
of $\kk$.

\section{Small Orders}
\label{sec:SmallOrders}
Here we look at the size of valences for several small orders.
\newline
{\tt Case $r=1$.}  Any positive integer is a valence, and the exact
shape of the polygons $\T_1[\kk]$ can be calculated easily by the
definition (see the first case of relation~\eqref{eqnewALL}). Thus we have
    \begin{equation}\label{eqFormula1}
	\begin{split}
          \sum_{\substack{\kk\in\A_1\\ 1\le k_1\le K}}1 & = K,
          \quad\text{for }K\ge 1\,.
	\end{split}
    \end{equation}
{\tt Case $r=2$.}
Let $k$ and $l$ be two consecutive valences and suppose they are
generated by $q_1,q_2,q_3,q_4$, that is,
	\begin{equation}\label{eqL101}
		kq_2=q_1+q_3\,, \qquad
		lq_3=q_2+q_4\,.
	\end{equation}

\begin{lemma}\label{Lemma100}
The smallest of any two consecutive neighbor valences cannot be
larger than~$3$.
\end{lemma}
\begin{proof}
By \eqref{eqL101} and the fact that the sum of consecutive
denominators of Farey fractions in $\FQ$ is larger than $Q$, it follows that
	\begin{equation*}
	\min(k,l)Q< \min(k,l)(q_2+q_3)\le kq_2+lq_3
        =(q_1+q_3)+(q_2+q_4)\le 4Q\,,
	\end{equation*}
which gives $\min(k,l)\le 3$, as required.
\end{proof}

\begin{lemma}\label{Lemma101}
There are no two neighbor valences both equal to $1$.
\end{lemma}
\begin{proof}
Let $q_1,q_2,q_3,q_4$ be consecutive denominators in $\FQ$, for some
$Q$, and assume that \eqref{eqL101} holds with $k=l=1$. Then, 
adding the two relations, we obtain $q_1+q_4=0$, a
contradiction which completes the proof of the lemma.
\end{proof}

\begin{lemma}\label{Lemma102}
Let $(k,l)$ be two neighbor valences in $\kall$. Then, if
one of $k$ or $l$ is $\ge 5$, than the other is equal to $1$.
\end{lemma}
\begin{proof}
Since the pairs $(k,l)$ and $(l,k)$ are either both  admissible or
not, we may assume that $k\ge 5$. Let $d_1$ be the bottom edge of the
quadrilateral $\T_1[k]$, and let $d_2$ be the top line of the strip
that should intersect $\T_1[k]$ in order to have a nonempty
$\T_2[k,l]$.

Our aim is to show that, for any $l\ge 2$, in the triangle $\T$, the
line $d_1$ with equation $y=\frac{lx+1}{kl-1}$ is under the line
$d_2$, whose equation is $y=\frac{x+1}{k+1}$. On $x=1$ this is true
since $(l+1)/(kl-1)\le 2/(k+1)$, which follows by our assumption that
$4\le (k-1)(l-1)$.  

The slope of $d_2$ is greater than
that of $d_1$, and this completes the proof of the lemma.
\end{proof}

Inspecting all the pairs $(k_1,k_2)$ in the remaining cases,
 one finds the set of pairs of neighbor valences presented
in Table~\ref{Table0} .
In particular, this gives
    \begin{equation}\label{eqFormula2}
	\begin{split}
	       \sum_{\substack{\kk\in\A_2\\ 1\le k_1,k_2\le K}}1 & 
			= 2K+3, 
				\quad\text{for }K\ge 4\,.\\
	\end{split}
    \end{equation}

{\tt Case $r=3$.}
The larger the order, the larger the noise, in other words, many
triples of consecutive valences occur. We check first only the end
points of a triple.
\begin{lemma}\label{Lemma103}
Let $(k,l,m)$ be three consecutive valences. Then, $\min(k,m)< 8$.
\end{lemma}
\begin{proof}
Suppose $k$, $l$ and $m$ are generated by $q_1,q_2,q_3,q_4,q_5$, that is,
	\begin{align}
		kq_2&=q_1+q_3, \label{102a}\\
		lq_3&=q_2+q_4\,,\label{102b}\\
		mq_4&=q_3+q_5\,.\label{102c}
	\end{align}

We split the argument in three parts.

{\tt Case 1.} Suppose $mq_4-kq_3\ge 0$. By \eqref{102a} and \eqref{102c} we
obtain $kq_2+mq_4\le 4Q$. Then 
	\begin{equation*}
	kQ<k(q_2+q_3)\le kq_2+kq_3+mq_4-kq_3
           <(q_1+q_3)+(q_3+q_5)\le 4Q\,,
	\end{equation*}
which gives that $k<4$.

{\tt Case 2.} Suppose $kq_2-mq_3\ge 0$. By symmetry, or proceeding
similarly as in {\tt Case 1}, it follows that $l<4$.

{\tt Case 3.} Now assume that $mq_4-kq_3< 0$ and $kq_2-mq_3< 0$.
Then
	\begin{align*}
		mQ&<m(q_3+q_4)<(k+m)q_3,\\
		kQ&<k(q_2+q_3)<(k+m)q_3,
	\end{align*}
which give $Q<2q_3$. Then, by \eqref{102b},
$Q/2<q_3\le lq_3=q_2+q_4$. This implies
	\begin{equation*}
	\min(k,m)Q/2< \min(k,m)(q_2+q_4)\le kq_2+mq_4\le 4Q\,,
	\end{equation*}
that is, $\min(k,m)< 8$, as claimed.
\end{proof}

On combining Lemma~\ref{Lemma103} with \eqref{eqFormula2} and the
analysis for the remaining triples summarized in Table~\ref{Table0},
we obtain \eqref{eqFormula}
for $r=3$ with the error term $C(3) = 15$:
    \begin{equation}\label{eqFormula3}
	\begin{split}
	       \sum_{\substack{\kk\in\A_3\\ 1\le k_1,k_2,k_3\le K}}1 & 
			= 3K+15, 
				\quad\text{for }K\ge 4\,.\\
	\end{split}
    \end{equation}

\smallskip

\small\scriptsize\footnotesize
\begin{center}
   {\sc Table 1.} Chains of valences. In the second column only one of $\kk$ and its reverse
\reflectbox{$\kk$}
is included.
\end{center}

\vspace*{-2.6mm}

\setlength{\doublerulesep}{1pt}
\setlongtables
\begin{longtable}{CL}
\hline
r & \mathrm{Chains\  of\  admissible\ valences\ of\ length\ } r \\ \hhline{|--|}
\endfirsthead
\multicolumn{2}{l}{\small\sl continued from previous page}\\ \hline
\endhead
\hline
\multicolumn{2}{r}{\small\sl continued on next page} \\ 
\endfoot
\hline
\endlastfoot
1 & (k) \ \text{for } k\ge 1 \\ \hhline{|--|} 
2 & (1,k) \ \text{for } k\ge 2;\ \ (2,2);\; (2,3);\; (2,4) \\ \hhline{|--|}
3 & (1,k,1) \ \text{for } k\ge 3;\ \ 
    (2,1,k) \ \text{for } k\ge 6; \ \ (2,2,2);\; (2,3,2);\; (4,1,4);\ \
    \\\hhline{||}
  &  (1,2,2);\; (1,2,3);\; (1,2,4);\; 
   (1,3,2);\; (1,4,2);\; (2,2,3);\;  \\\hhline{||}
   &(3,1,4);\; (3,1,5);\; (3,1,6);\; (3,1,7);\; (3,1,8);\; 
   (4,1,5)\\ \hhline{|--|}
4 & (1,k,1,2) \ \text{for } k\ge 6;\
    (2,2,1,k) \ \text{for } k\ge 10;\ \ (2,2,2,2); \\\hhline{||}
 &(1, 2, 2, 2);\;(1, 2, 2, 3);\;(1, 2, 3, 1);\;(1, 2, 3, 2);\;
  (1, 2, 4, 1);\;(1, 3, 2, 2);  \\\hhline{||}
 &(1, 3, 1, 5);\;(1, 3, 1, 6);\;(1, 3, 1, 7);\;(1, 3, 1, 8);\; 
  (1, 4, 1, 4);\;(1, 4, 1, 5);\;(1, 5, 1, 4); \\\hhline{||}
 &(1, 4, 1, 3);\;(1, 5, 1, 3);\;(1, 6, 1, 3);\;  
  (1, 7, 1, 3);\;(1, 8, 1, 3);\\\hhline{||}
 &(2, 2, 2, 3);\;(2, 2, 3, 2);\; (2, 3, 1, 4);\;(2, 3, 1, 5);\;(2, 3, 1, 6);
  (2, 4, 1, 3);\;(2, 4, 1, 4);\\\hhline{||}
 &(3, 2, 1, 7);\;(3, 2, 1, 8);\;
  (3, 2, 1, 9);\;(3, 2, 1, 10);\;(3, 2, 1, 11);\;(3, 2, 1, 12); \\\hhline{||}
 &(4, 2, 1, 6);\;(4, 2, 1, 7);\;(4, 2, 1, 8) 
%
\label{Table0} 
\end{longtable}

\normalsize

\section{Completion of the Proof of Theorem~\ref{TheoremAr}}\label{sec:Proof}
By induction, we show that at most one component of an admissible $r$-tuple can be
excessively large.
Thus, for a given $\kk=(k_1,\dots,k_r)\in\A_r$ 
we have to show that the minimum of $k_1$ and $k_r$ can not exceed a certain margin, while bounds
for the other components $k_2,\dots,k_{r-1}$ follow by the induction
hypothesis. 
For this it is helpful to see that $\T_r[\kk]$ lies at the
intersection between $\T_{r-1}[k_1,\dots,k_{r-1}]$ and the angular region
defined by the last condition in the definition of $\T_r[\kk]$:
	\begin{equation*}
		\V(k_1,\dots,k_r):=\Big\{ (x,y)\in \T \colon\ 
              k_r=\Big[ \tfrac{1+x_{r-2}}{x_{r-1}} \Big] \Big\} \,,
	\end{equation*}
in which the $x_j=x_j(k_1,\dots,k_j)$, for $j\ge 1$,  are defined by
\eqref{eqL3}. We claim that if both $k_1$ and $k_r$ were large enough,
then the intersection $\T_{r-1}[k_1,\dots,k_{r-1}]\cap\V(k_1,\dots,k_r)$ is
empty. 
This would imply $\kk\not\in\A_r$, contradicting our assumption.

The main point of the proof is to show more than it is required. Namely, we
shall show that even the superset $\T_{1}[k_1]\cap\V(k_1,\dots,k_r)$ is
empty when $k_1,k_r$ both surpass a certain magnitude. We do this by
proving that the angle $\V(k_1,\dots,k_r)$ lies under $\T_1[k_1]$.

From \eqref{eqtes} and \eqref{eqL4} we know that for $k_1\ge 2$, the
line $d'$, the bottom edge of quadrangle  $\T_1[k_1]$ has equation
$y=(x+1)/(k_1+1)$, and $d''$, the top edge of $\V(k_1,\dots,k_r)$ has
equation $y=\frac{p_{r-1}(k_2,\dots,k_r)x+1}{p_r(k_1,\dots,k_r)}$.
The argument has two parts. Firstly we see that in our hypotheses,
the slope of $d''$ is greater than the slope of $d'$ and secondly, we
check the position of the points of intersection of $d'$ and
respectively $d''$ with the vertical line $\{x=1\}$.

Let $m',m''$ be the slopes of $d',d''$ respectively. Then, $m''>m'$ is
equivalent to
	\begin{equation}\label{eqThmp}
		k_1p_{r-1}(k_2,\dots,k_r)>p_r(k_1,\dots,k_r)\,.
	\end{equation}
Here, by \eqref{eqR} and by the symmetry property \eqref{eqSimmetry},
the right-hand side is 
	\begin{equation}\label{eqThmpp}
	   \begin{split}
		p_r(k_1,\dots,k_r)
		&=p_r(k_r,\dots,k_1)\\
		&=k_1p_{r-1}(k_r,\dots,k_2)-p_{r-2}(k_r,\dots,k_3)\\
		&=k_1p_{r-1}(k_2,\dots,k_r)-p_{r-2}(k_3,\dots,k_r)\,.
           \end{split}
	\end{equation}
Inserting \eqref{eqThmpp} into \eqref{eqThmp} and reducing the terms,
one finds that the inequality $m''>m'$ is equivalent to
$p_{r-2}(k_3,\dots,k_r)>0$, which is always true for $(k_3,\dots,k_r)\in\A_{r-2}$.

Now, for the second part of the argument, let $A$ and $B$ be the points of intersection of $d',d''$
with $\{x=1\}$
respectively, that is, $\{A\}=d'\cap\{x=1\}$ and  $\{B\}=d''\cap\{x=1\}$. It
remains to show that $B$ lies under $A$, which is the same as showing
that
	\begin{equation}\label{eqThmq}
		(k_1+1)\big(p_{r-1}(k_2,\dots,k_r)+1\big)
		<p_r(k_1,\dots,k_r)\,.
	\end{equation}
In order to make apparent the influence of $k_1$ and $k_r$ in this
inequality, we extract them by reducing the order. This is done by
using several times \eqref{eqR}, as in \eqref{eqThmpp}. Then
\eqref{eqThmq} reduces to the following inequality
	\begin{multline*}
	\qquad
	k_1p_{r-3}(k_2,\dots,k_{r-2})+k_r\Big(
		2p_{r-3}(k_3,\dots,k_{r-1})+p_{r-2}(k_2,\dots,k_{r-1})\Big)\\
	<k_1k_rp_{r-2}(k_2,\dots,k_{r-1})
		+p_{r-3}(k_2,\dots,k_{r-2})+2p_{r-4}(k_3,\dots,k_{r-2})\,.
	\qquad
	\end{multline*}
Here, since $k_2,\dots,k_{r-1}$ are bounded, the inequality becomes true as
soon as both $k_1$ and $k_r$ get larger than a certain quantity, so
$B$ lies under $A$.

In conclusion, an admissible tuple has at most one very large component.
On the other hand there are many possible combinations that consist of small numbers that may
form a subsequence of an admissible tuple. Moreover, when the components of $\kk$ follow a regular
pattern, the vertices of polygons $\T_r[\kk]$ can be
expressed in closed formulas. Such a pattern is $ \dots, 1,4,1,4,\dots$, but the meaningful
example is the constant sequence of $2$s that appear in the neighborhood of a large peak.
The formulas recorded in the following proposition are obtained by recording the data, step
by step, during an induction process that resembles the one described above.


%
\begin{proposition}\label{Lemma107}
Fix $s\ge 0$ and $t\ge 0$. Then there exists a positive integer $k_0$ depending on $s$ and $t$
only, such that for any integer $k\ge k_0$ 
the quadrangle 
\vspace*{-4.5mm}
$\T_{s+1+t}[\underbrace{2,\dots,2,1}_{s\ \rm{components}}\!\!, k,
        \underbrace{1,2,\dots,2}_{t\ \rm{components}}]$
has vertices  given by:
    \begin{equation}\label{eqnewALL}
      \begin{cases}
           & \Big\{ \Big(\frac{k}{k+2},\frac{2}{k+2}\Big);\,
        \Big(\frac{k+1}{k+1},\frac{2}{k+1}\Big);\,
        \Big(\frac{k}{k},\frac{2}{k}\Big);\,
        \Big(\frac{k-1}{k+1},\frac{2}{k+1}\Big)
        \Big\},\, \quad\rm{for}\ s=0,\\
&\mbox{}\\
& \Big\{ \Big(\frac{k-2s}{k+2},\frac{k-2s+2}{k+2}\Big);\,
        \Big(\frac{k-2s+1}{k+1},\frac{k-2s+3}{k+1}\Big);\,\\
        &\qquad\qquad\quad\ \Big(\frac{k-2s}{k},\frac{k-2s+2}{k}\Big);\,
        \Big(\frac{k-2s-1}{k+1},\frac{k-2s+1}{k+1}\Big)\Big\}\,,
        \quad\rm{for}\ s\ge 1\,.
      \end{cases}
    \end{equation}

%
  \end{proposition}
Notice the two 'attractors'  $(1,0)$ and $(1,1)$ of the shrinking quadrangles $\T_{r}[\kk]$ with one
component large.
They are indicated by the first and second case of relation~\eqref{eqnewALL}, respectively.
%
%
In particular, since  polygons $\T_{r}[k,*]$ are subsets of $\T_{1}[k]$,
Proposition~\ref{Lemma107} shows that when a component $k$ of an admissible tuple $\kk$
is large enough, it should be followed by $1$, and next, the more distant close neighbors 
should be $2$s. By symmetry, this pattern identifies uniquely the components that precede the very
large component, also, and this concludes the proof of the theorem.




%

\smallskip
\noindent
\textbf{Acknowledgement}
     The authors are grateful to the referee for his useful comments and suggestions.  

%

\end{document}